\documentclass[reqno]{amsart}
\usepackage{amsmath,amsthm,amscd,amssymb,amsfonts, amsbsy}
\usepackage{latexsym}
\usepackage{color}
\usepackage{enumerate}
\usepackage{pxfonts}

\theoremstyle{plain}
\newtheorem{theorem}[equation]{Theorem}
\newtheorem{lemma}[equation]{Lemma}

\newtheorem{proposition}[equation]{Proposition}

\theoremstyle{definition}

\theoremstyle{remark}
\newtheorem{remark}[equation]{Remark}

\newcommand{\dv}{\operatorname{div}}

\newcommand{\dist}{\operatorname{dist}}
\newcommand{\diam}{\operatorname{diam}}

\newcommand{\tr}{\operatorname{tr}}

\numberwithin{equation}{section}

\newcommand{\bR}{\mathbb{R}}

\providecommand{\set}[1]{\{#1\}}

\providecommand{\abs}[1]{\lvert#1\rvert}
\providecommand{\Abs}[1]{\left\lvert#1\right\rvert}

\providecommand{\norm}[1]{\lVert#1\rVert}

\begin{document}
\title[Estimates for Green function for non-divergence elliptic equation]
{Estimates for Green's functions of elliptic equations in non-divergence form with continuous coefficients}

\author[S. Kim]{Seick Kim}
\address[S. Kim]{Department of Mathematics, Yonsei University, 50 Yonsei-ro, Seodaemun-gu, Seoul 03722, Republic of Korea}
\email{kimseick@yonsei.ac.kr}
\thanks{S. Kim is partially supported by National Research Foundation of Korea (NRF) Grant No. NRF-2019R1A2C2002724 and No. NRF-20151009350.}

\author[S. Lee]{Sungjin Lee}
\address[S. Lee]{Department of Mathematics, Yonsei University, 50 Yonsei-ro, Seodaemun-gu, Seoul 03722, Republic of Korea}
\email{sungjinlee@yonsei.ac.kr}



\begin{abstract}
We present a new method for the existence and pointwise estimates of a Green's function of non-divergence form elliptic operator with Dini mean oscillation coefficients.
We also present a sharp comparison with the corresponding Green's function for constant coefficients equations.
\end{abstract}
\maketitle

\section{Introduction and main results}
We consider a second-order elliptic operator $L$ in non-divergence form
\begin{equation}							\label{master-nd}
L u=  a^{ij}(x) D_{ij} u,
\end{equation}
where the coefficient $\mathbf{A}:=(a^{ij})$ are symmetric and satisfy the uniform ellipticity condition. Namely
\begin{equation}					\label{ellipticity-nd}
a^{ij}=a^{ji},\quad \lambda \abs{\xi}^2 \le a^{ij}(x) \xi^i \xi^j \le \Lambda \abs{\xi}^2,
\end{equation}
for some positive constants $\lambda$ and $\Lambda$ in a domain $\Omega \subset \bR^n$ with $n \ge 3$.
Here and below, we use the usual summation convention over repeated indices.

In this article, we are concerned with construction and pointwise estimates for the Green's function $G(x,y)$ of the non-divergent operator $L$ in $\Omega$.
In a recent article \cite{HK20}, it is shown that if the coefficients $\mathbf{A}$ is of Dini mean oscillation and the domain $\Omega$ is bounded and has $C^{2,\alpha}$ boundary, then the Green's function exists and satisfies the pointwise bound
\begin{equation}				\label{bound1}
\abs{G(x,y)} \le C \abs{x-y}^{2-n}.
\end{equation}
Before proceeding further, let us introduce the definition of Dini mean oscillation. 
For $x\in \bR^n$ and $r>0$, we denote by $B(x,r)$ the Euclidean ball with radius $r$ centered at $x$, and write $\Omega(x,r):=\Omega \cap B(x,r)$.
We denote
\[
\omega_{\mathbf A}(r, x):= \fint_{\Omega(x,r)} \,\abs{\mathbf{A}(y)-\bar {\mathbf A}_{\Omega(x,r)}}\,dy, \quad \text{ where } \;\bar{\mathbf A}_{\Omega(x,r)} :=\fint_{\Omega(x,r)} \mathbf{A},
\]
and we write
\begin{equation}			\label{def_mo}
\omega_{\mathbf A}(r, D):= \sup_{x \in D} \omega_{\mathbf A}(r, x) \quad \text{and}\quad \omega_{\mathbf A}(r)=\omega_{\mathbf A}(r, \bar \Omega).
\end{equation}
We say that $\mathbf{A}$ is of Dini mean oscillation in $\Omega$ if $\omega_{\mathbf A}(r)$ satisfies the Dini condition; i.e.,
\[
\int_0^1 \frac{\omega_{\mathbf A}(t)}t \,dt <+\infty.
\]
It is clear that if $\mathbf{A}$ is Dini continuous, then $\mathbf{A}$ is of Dini mean oscillation.
Also if $\mathbf{A}$ is of Dini mean oscillation, then $\mathbf{A}$ is uniformly continuous in $\Omega$ with its modulus of continuity controlled by $\omega_\mathbf{A}$; see \cite[Appendix]{HK20}.
However, a function of Dini mean oscillation is not necessarily Dini continuous; see \cite{DK17} for an example.

The main result of \cite{HK20} is interesting because unlike the Green's function for uniformly elliptic operators in divergence form, the Green's function for non-divergent elliptic operators does not necessarily enjoy the pointwise bound \eqref{bound1} even in the case when the coefficient $\mathbf A$ is uniformly continuous; see \cite{Bauman84}.
It should be noted that the Dini mean oscillation condition is the weakest assumption in the literature that guarantees the pointwise bound \eqref{bound1}.
The proof in \cite{HK20} is based on considering approximate Green's functions (as in \cite{GW82, HK07}) and showing that they satisfy specific estimates, as well as a local $L^\infty$ estimate for solutions to the adjoint equation $L^\ast u=0$, which is shown in \cite{DK17, DEK18}.
This $L^\infty$ estimate is crucial for the pointwise bound \eqref{bound1} and it is where the Dini mean oscillation condition is strongly used; a mere continuity of $\mathbf A$  is not enough to produce such an estimate.
We should recall that the adjoint operator $L^\ast$ is given by
\[
L^\ast u = D_{ij} (a^{ij}(x) u).
\]
We should also mention that there are many papers in the literature dealing with the existence and estimates of Green's functions or fundamental solutions of non-divergence form elliptic or parabolic operators with measurable or continuous coefficients; see e.g. \cite{Bauman84b, Bauman85, FS84, FGMS88, Krylov92, Esc2000, Cho11}.

In this article we give an alternative proof for the existence of Green's function.
More precisely, we construct Green's function from that of the corresponding constant coefficients operator resulting from ``freezing coefficients''.
We shall use $L^p$ theory for the adjoint operator in this process.
We then utilize the local $L^\infty$ estimates for adjoint solutions established in \cite{DK17, DEK18} to get the pointwise bound \eqref{bound1} for the Green's function.
One prominent advantage of this approach is that it yields a sharp comparison with the Green's function for constant coefficients operator.
In particular, we shall show that
\begin{equation}				\label{asymptotics}
G(x_0,x)-G_0(x_0, x)= o(\abs{x-x_0}^{2-n})\;\text{ as }\; x\to x_0,
\end{equation}
where $G_0$ is the Green's function of the constant coefficient operator $L_0$ given by
\begin{equation}				\label{eq2203sun}
L_0u:= a^{ij}(x_0) D_{ij} u,
\end{equation}
provided that the mean oscillation of $\mathbf A$ satisfies so-called ``double Dini condition'' near $x_0$; that is, we have
\[
\int_0^1 \frac{1}{s} \int_0^s \frac{\omega_{\mathbf A}(t,\Omega(x_0, r_0))}{t}\,dt\,ds = \int_0^1 \frac{\omega_{\mathbf A}(t,\Omega(x_0, r_0))\ln \frac{1}{t}}{t}\,dt <+\infty,
\]
for some $r_0>0$.
The asymptotic behavior \eqref{asymptotics} is well known for the Green's functions for elliptic operators in divergence form with continuous coefficients; see \cite{DM95}.
However, in the non-divergence form setting, this is a new result and it is one of the novelties in our work.
We now present our main theorem.

\begin{theorem}					\label{thm01}
Let $\Omega$ be a bounded $C^{2,\alpha}$ domain in $\bR^n$ with $n\ge 3$.
Assume the coefficient $\mathbf{A}=(a^{ij})$ of the operator $L$ in \eqref{master-nd} satisfies the uniform ellipticity condition \eqref{ellipticity-nd} and is of Dini mean oscillation in $\Omega$.
Then, there exists a unique Green's function $G(x,y)$ of the operator $L$ in $\Omega$ and it satisfies the pointwise estimate
\[
\abs{G(x,y)} \le C \abs{x-y}^{2-n},
\]
where $C=C(n,\lambda, \Lambda, \Omega, \omega_{\mathbf A})$.
Moreover, if there is some $r_0>0$ such that $\omega_{\mathbf A}(t, \Omega(x_0, r_0))$ satisfies double Dini condition
\begin{equation}				\label{eq1243f}
\int_0^1 \frac{1}{s} \int_0^s \frac{\omega_{\mathbf A}(t,\Omega(x_0, r_0))}{t}\,dt\,ds = \int_0^1 \frac{\omega_{\mathbf A}(t,\Omega(x_0, r_0))\ln \frac{1}{t}}{t}\,dt <+\infty,
\end{equation}
then we have
\begin{equation}				\label{eq1725sun}
\lim_{x\to x_0}\, \abs{x-x_0}^{n-2}\,\abs{G(x_0, x)-G_0(x_0, x)}=0,
\end{equation}
where $G_0$ is the Green's function of the constant coefficient operator $L_0$ as in \eqref{eq2203sun}.
\end{theorem}

\begin{remark}
As stated in \cite{HK20}, pointwise estimates for $D_x G(x,y)$ and $D_x^2 G(x,y)$ are also available.
They are obtained from \eqref{bound1} via local $L^\infty$ estimates for first and second derivatives of solutions to $Lu=0$ as established in \cite{DK17, DEK18}.
We only treat the case when $n\ge 3$ in this article and we refer to \cite{DK20} for two dimensional case.
In a separate paper \cite{KL21p}, we construct the fundamental solution for parabolic equations in non-divergence form with Dini mean oscillation coefficients and establish Gaussian bounds for the fundamental solution.
\end{remark}

\section{Preliminary lemmas}
In this section, we present some technical lemmas which will be used in the proof of Theorem~\ref{thm01}.
We need to consider the boundary value problem of the form
\begin{equation}				\label{adj_eq}
L^\ast v = \dv^2 \mathbf{g} + f\;\text{ in }\;\Omega,\quad v=\frac{\mathbf{g} \nu\cdot \nu}{\mathbf{A}\nu\cdot \nu}\;\text{ on }\;\partial \Omega,
\end{equation}
where $\mathbf{g}=(g^{ij})$ is an $n \times n$ matrix-valued function,
\[
\dv^2 \mathbf{g}:=D_{ij}g^{ij},
\]
and $\nu$ is the unit exterior normal vector of $\partial\Omega$.
For $\mathbf{g} \in L^p(\Omega)$ and $f \in L^p(\Omega)$, where $1<p<\infty$ and $\frac{1}{p}+ \frac{1}{p'}=1$, we say that $v$ in $L^p(\Omega)$ is an adjoint solution of \eqref{adj_eq} if $v$ satisfies
\begin{equation}			\label{eq1519m}
\int_\Omega v Lu = \int_\Omega \tr(\mathbf{g} D^2u) + \int_\Omega f u
\end{equation}
for any $u$ in $W^{2,p'}(\Omega) \cap W^{1,p'}_0(\Omega)$.

\begin{lemma}				\label{lem01}
Let $1<p<\infty$ and assume that $\mathbf{g} \in L^p(\Omega)$ and $f \in L^p(\Omega)$.
Then there exists a unique adjoint solution $u$ in $L^p(\Omega)$.
Moreover, the following estimates holds.
\[
\norm{u}_{L^p(\Omega)} \le C \left( \norm{\mathbf g}_{L^p(\Omega)} + \norm{f}_{L^p(\Omega)} \right),
\]
where a constant $C$ depends on $\Omega$, $p$, $n$, $\lambda$, $\Lambda$, and $\omega_{\mathbf A}$.
\end{lemma}
\begin{proof}
See  \cite[Lemma~2]{EM2016}.
\end{proof}

The proof of next lemma is implicitly given in the proof of \cite[Theorem~1.10]{DK17}.
However, an estimate like \eqref{eq1023m} does not appear explicitly in the literature and  we provide a proof in the Appendix for reader's convenience.
It should be emphasized that the lemma asserts that to get a local $L^\infty$ estimate of the solution $u$, only a local information on the data $\mathbf{g}$ is needed.

\begin{lemma}			\label{prop01}
Let $R_0>0$ and $\mathbf{g}=(g^{ij})$ be of Dini mean oscillation in $B(x_0, R_0)$.
Suppose $u$ is an $L^2$ solution of 
\[
L^\ast u= \dv^2 \mathbf{g}\;\text{ in }\;B(x_0, 2R),
\]
where $0<R \le \frac12 R_0$.
Then we have
\begin{equation}					\label{eq1023m}
\norm{u}_{L^\infty(B(x_0,R))} \le C \left(\fint_{B(x_0,2R)} \abs{u} + \int_0^{R} \frac{\omega_{\mathbf g}(t, B(x_0,2R))}{t}\,dt \right),
\end{equation}
where $C=C(n,\lambda, \Lambda, \omega_{\mathbf A}, R_0)$.
\end{lemma}
\begin{proof}
See Appendix.
\end{proof}

The next lemma is an extension of Lemma~\ref{prop01} up to the $C^{2,\alpha}$ boundary.
\begin{lemma}			\label{prop02}
Let $\Omega$ be a bounded $C^{2,\alpha}$ domain.
Assume that $\mathbf{g}=(g^{ij})$ are of Dini mean oscillation in $\Omega$.
Let $u \in L^2(\Omega)$ be the solution of the adjoint problem
\[
L^\ast u= \dv^2 \mathbf{g}\;\text{ in }\;\Omega, \quad 
u=\frac{\mathbf{g} \nu\cdot \nu}{\mathbf{A}\nu\cdot \nu}\;\text{ on }\;\partial \Omega.
\]
Then for $x_0 \in \bar \Omega$ and $0<R \le \frac12 \diam \Omega$, we have
\[
\norm{u}_{L^\infty(\Omega(x_0,R))} \le C \left(\fint_{\Omega(x_0,2R)} \abs{u} + \int_0^{R} \frac{\omega_{\mathbf g}(t, \Omega(x_0,2R))}{t}\,dt \right),
\]
$C=C(n,\lambda, \Lambda, \omega_{\mathbf A}, \Omega)$.
\end{lemma}
\begin{proof}
By flattening the boundary, it suffices to get an estimate in half balls that corresponds to \eqref{eq1023m}.
It is obtained by replicating the proof of \cite[Lemma~2.26]{DEK18} in the same fashion as \eqref{eq1023m} is derived.
We leave the details to the readers.
\end{proof}

\section{Proof of Theorem~\ref{thm01}}
The organization of the proof is as follows.
In Sec.~\ref{sec3.1}, we first construct the Green's function $G^\ast(x,y)$ for the adjoint operator.
In Sec.~\ref{sec3.2} and \ref{sec3.3}  it will be shown that the adjoint Green's function has the pointwise bound $\abs{G^\ast(x,y)} \le C\abs{x-y}^{2-n}$.
In Sec.~\ref{sec3.4}, we show that $G(x,y)=G^\ast(y,x)$ becomes the Green's function and thus it also has the pointwise bound $\abs{G(x,y)} \le C \abs{x-y}^{2-n}$.
Finally, in Sec.~\ref{sec3.5} we establish the asymptotic formula \eqref{eq1725sun}. 

\subsection{Construction of adjoint Green's function}			\label{sec3.1}
Let $x_0 \in \Omega$ be fixed and denote
\[
\mathbf{A}_0=\mathbf{A}(x_0) \quad \text{and}\quad L_0 u := a^{ij}(x_0) D_{ij}u= \tr(\mathbf{A}_0 D^2 u).
\]
Let $G_0(x,y)$ be the Green's function for $L_0$ in $\Omega$.
Since $L_0$ is an elliptic operator with constant coefficients, the existence of $G_0$ as well as the following pointwise bound is well known.
\begin{equation}				\label{eq1224th}
\abs{G_0(x,y)} \le  C\abs{x-y}^{2-n}\quad \quad (x\neq y),
\end{equation}
where $C=C(n, \lambda, \Lambda)$.
Moreover, since $\mathbf{A}_0$ is symmetric, we have $L_0=L_0^\ast$ and $G_0$ is also symmetric, i.e.,
\[
G_0(x,y)=G_0(y,x) \quad (x\neq y).
\]

We shall now construct $G^\ast(\cdot, x_0)$, Green's function for $L^\ast$ in $\Omega$ with a pole at $x_0$.
Formally, we would have 
\[
L^\ast G^\ast(\cdot, x_0)=\delta(\cdot -x_0) \;\text{ in }\;\Omega,\quad G^\ast(\cdot, x_0)=0 \;\text{ on }\;\partial\Omega.
\]
On the other hand, since $L_0=L_0^\ast$, we have
\[
L_0^\ast G_0(\cdot, x_0)=\delta(\cdot -x_0) \;\text{ in }\;\Omega,\quad G_0(\cdot, x_0)=0 \;\text{ on }\;\partial\Omega.
\]
Therefore, if we set $v=G^\ast(\cdot, x_0)-G_0(\cdot, x_0)$, then we would have $v=0$ on $\partial\Omega$ and 
\begin{align*}
L^\ast v&=L^\ast G^\ast(\cdot,x_0)-L^\ast G_0(\cdot, x_0)+L_0^\ast G_0(\cdot, x_0) -L_0^\ast G_0(\cdot, x_0)\\
&=-(L^\ast-L_0^\ast)G_0(\cdot, x_0)=-\dv^2((\mathbf{A}-\mathbf{A}_0)G_0(\cdot, x_0))\quad\text{in }\;\Omega,
\end{align*}
which lead us to consider the problem
\begin{equation}				\label{eq1042th}
L^\ast v = \dv^2 \mathbf{g}\;\text{ in }\; \Omega,\quad 
v=\frac{\mathbf{g}\nu\cdot \nu}{\mathbf{A}\nu\cdot \nu} \;\text{ on }\;\partial \Omega,
\end{equation}
where we denote
\[
\mathbf{g}:=-(\mathbf{A}- \mathbf{A}_0) G_0(\cdot, x_0).
\]
Notice that since $\mathbf g$ vanishes on $\partial \Omega$, the  boundary condition in \eqref{eq1042th} reads simply that $v=0$ on $\partial\Omega$.

\begin{lemma}					\label{lem3.3}
For $\mathbf g=-(\mathbf{A}- \mathbf{A}_0) G_0(\cdot, x_0)$, we have $\mathbf{g} \in L^p(\Omega)$ for all $p \in [1, \frac{n}{n-2})$.
\end{lemma}
\begin{proof}
First, observe that we have
\begin{align}					\nonumber	
\fint_{\Omega(x_0, r)} \abs{\mathbf{A}-\mathbf{A}_0}
&\le \fint_{\Omega(x_0, r)} \abs{\mathbf{A}(x)-\bar{\mathbf A}_{\Omega(x_0,r)}}\,dx+ \abs{\bar{\mathbf A}_{\Omega(x_0,r)}-\mathbf{A}(x_0)}\\
							\label{eq1142tu}
&\le \omega_{\mathbf A}(r, x_0)+ C\int_0^r \frac{\omega_{\mathbf A}(t,x_0)}{t}\,dt \le C\int_0^r \frac{\omega_{\mathbf A}(t,x_0)}{t}\,dt,
\end{align}
where we used \cite[Appendix]{HK20} in the second line.
Next, using \eqref{eq1224th} and 
\[
\norm{\mathbf{A}-\mathbf{A}_0}_{L^\infty(\Omega)} \le C,
\]
where $C=C(n,\Lambda)$, we have for $1\le p <\frac{n}{n-2}$ that
\begin{align}				\nonumber
\int_\Omega \abs{\mathbf g}^p&=\sum_{k=0}^\infty \int_{\Omega(x_0, 2^{-k})\setminus \Omega(x_0, 2^{-k-1})} \abs{\mathbf g}^p + \int_{\Omega\setminus \Omega(x_0, 1)}\abs{\mathbf g}^p \\
						\nonumber
&\le C \sum_{k=0}^\infty 2^{(n-2)pk} \int_{\Omega(x_0, 2^{-k})\setminus \Omega(x_0, 2^{-k-1})} \abs{\mathbf{A}-\mathbf{A}_0} + C \int_{\Omega\setminus \Omega(x_0, 1)} \abs{x-x_0}^{(2-n)p}\,dx\\
						\nonumber
&\le C \sum_{k=0}^\infty 2^{(n-2)pk} 2^{-kn} \fint_{\Omega(x_0, 2^{-k})} \abs{\mathbf{A}-\mathbf{A}_0} + C (\diam\Omega)^{(2-n)p+n}\\
						\label{eq1041th}
&\le  C\left(\int_0^1 \frac{\omega_{\mathbf A}(t,x_0)}{t}\,dt\right) \sum_{k=0}^\infty 2^{\{(n-2)p-n\}k}+ C (\diam\Omega)^{(2-n)p+n}<+\infty,
\end{align}
where we used \eqref{eq1142tu} in the last line.
\end{proof}

By Lemmas~\ref{lem01} and \ref{lem3.3}, we find that there exists a unique solution $v$ of  the problem \eqref{eq1042th} and $v \in L^p(\Omega)$ for all $p \in (1, \frac{n}{n-2})$.

Now, We claim that $G^\ast(\cdot, x_0)$ defined as
\begin{equation}				\label{eq1531m}
G^\ast(\cdot, x_0):=G_0(\cdot, x_0) + v
\end{equation}
becomes Green's function of $L^\ast$ in $\Omega$ with a pole at $x_0$.
Indeed, for any $f\in L^{p'}(\Omega)$ with $p'>\frac{n}{2}$, let $u \in W^{2,p'}(\Omega)\cap W^{1,p'}_0(\Omega)$ be the strong solution of
\begin{equation}				\label{eq2059m}
Lu=f \;\text{ in }\;\Omega,\quad u=0\;\text{ on }\;\partial \Omega.
\end{equation}
Then, by \eqref{eq1519m} and \eqref{eq1042th}, we have
\[
\int_\Omega v f= \int_\Omega G_0(\cdot, x_0) L_0 u -\int_\Omega G_0(\cdot, x_0) f=u(x_0)-\int_\Omega G_0(\cdot, x_0)f,
\]
where we use the fact that $G_0$ is the Green's function for the operator $L_0$.
Therefore, by \eqref{eq1531m}, we have
\[
u(x_0)= \int_\Omega G^\ast(\cdot, x_0) f,
\]
which means that $G^\ast(x, y)$ is the Green's function for the adjoint operator $L^\ast$.
See \cite[Remark 1.14]{HK20}.
As a matter of fact, we proved the following.
\begin{proposition}				\label{prop3.9}
For $p>\frac{n}{2}$ and $f\in L^{p}(\Omega)$, if $u \in W^{2,p}(\Omega)\cap W^{1,p}_0(\Omega)$ be the strong solution of \eqref{eq2059m}, then we have the representation formula
\begin{equation}				\label{eq2037m}
u(x)= \int_\Omega G^\ast(y, x) f(y)\, dy.
\end{equation}
\end{proposition}

\subsection{Pointwise estimates of adjoint Green's functions}			\label{sec3.2}
In this section, we shall establish
\begin{equation}				\label{eq0927f}
\abs{G^\ast(x,x_0)} \le C \abs{x-x_0}^{2-n}\quad\text{ in }\;\Omega \setminus \set{x_0}. 
\end{equation}
Let $v$ be as in \eqref{eq1042th} and define $\mathbf{g}_1$ and $\mathbf{g}_2$ by
\begin{equation*}
\mathbf{g}_1= -\zeta(\mathbf{A}- \mathbf{A}_0) G_0(\cdot, x_0)\quad\text{and}\quad \mathbf{g}_2= -(1-\zeta)(\mathbf{A}-\mathbf{A}_0) G_0(\cdot, x_0),
\end{equation*}
where $\zeta$ is a smooth function on $\bR^n$ such that
\begin{equation*}			
0\leq \zeta \leq 1, \quad \zeta=0 \;\text{ in }\;B(x_0,r),\quad \zeta=1\;\text{ in }\bR^n\setminus B(x_0,2r),\quad \abs{D\zeta} \le 2/r,
\end{equation*}
and $r>0$ is to be fixed later.

We note that $\mathbf{g}_1 \in L^{p_1}(\Omega)$ for any $p_1>\frac{n}{n-2}$ and $\mathbf{g}_2 \in L^{p_2}(\Omega)$ for any $p_2 <\frac{n}{n-2}$.
Indeed, the computation in \eqref{eq1041th} reveals that for $p_1>\frac{n}{n-2}$  we have
\begin{equation}				\label{eq1539th}
\int_\Omega \abs{\mathbf{g}_1}^{p_1} \le C \int_{\Omega\setminus \Omega(x_0,r)} \abs{x-x_0}^{(2-n)p_1}\,dx \le C r^{(2-n)p_1+n}
\end{equation}
and for $p_2 <\frac{n}{n-2}$ we have
\begin{equation}				\label{eq1540th}
\norm{\mathbf{g}_2}_{L^{p_2}(\Omega)} \le C\left(\int_0^{2r} \frac{\omega_{\mathbf A}(t,x_0)}{t}\,dt\right)^{1/p_2} r^{2-n+n/p_2}.
\end{equation}
Fix a $p_1 \in (\frac{n}{n-2},\infty)$ and let us write $v=v_1+v_2$, where $v_1 \in L^{p_1}(\Omega)$ is the solution of
\begin{equation}				\label{eq1226th}
L^\ast v_1 = \dv^2 \mathbf{g}_1\;\text{ in }\; \Omega,\quad 
v_1=\frac{\mathbf{g}_1\nu\cdot \nu}{\mathbf{A}\nu\cdot \nu} \;\text{ on }\;\partial \Omega.
\end{equation}
Then by Lemma~\ref{lem01} and \eqref{eq1539th}, we have
\begin{equation}				\label{eq1618th}
\norm{v_1}_{L^{p_1}(\Omega)} \le C r^{2-n+n/p_1}.
\end{equation}
On the other hand, note that $v_2=v-v_1$ satisfies
\begin{equation}				\label{eq1244th}
L^\ast v_2 = \dv^2 \mathbf{g}_2 \;\text{ in }\; \Omega,\quad 
v_2=\frac{\mathbf{g}_2 \nu\cdot \nu}{\mathbf{A}\nu\cdot \nu} \;\text{ on }\;\partial \Omega.
\end{equation}
By Lemma~\ref{lem01} and \eqref{eq1540th}, for $p_2 \in (1, \frac{n}{n-2})$, we have $v_2 \in L^{p_2}(\Omega)$ with 
\begin{equation}				\label{eq1541th}
\norm{v_2}_{L^{p_2}(\Omega)} \le C\left(\int_0^{2r} \frac{\omega_{\mathbf A}(t,x_0)}{t}\,dt\right)^{1/p_2} r^{2-n+n/p_2}.
\end{equation}

Now, for any fixed $y_0 \in \Omega$ with $y_0 \neq x_0$, we take 
\[
r=\tfrac15 \abs{y_0-x_0}
\]
and estimate $v_1(y_0)$ by using Lemma~\ref{prop02} as follows. 
\begin{equation}		\label{eq2017th}
\abs{v_1(y_0)} \le \norm{v_1}_{L^\infty(\Omega(y_0, r))} \le  C \fint_{\Omega(y_0,2r)} \abs{v_1} + C \int_0^{r} \frac{\omega_{\mathbf{g}_1}(t, \Omega(y_0,2r))}{t}\,dt.
\end{equation}
By H\"older's inequality and \eqref{eq1618th}, we have
\begin{equation}		\label{eq2018th}
\fint_{\Omega(y_0,2r)} \abs{v_1} \le \left(\fint_{\Omega(y_0,2r)} \abs{v_1}^{p_1}\right)^{1/p_1} \le C r^{-n/p_1} \norm{v_1}_{L^{p_1}(\Omega)} \le C r^{2-n}.
\end{equation}

\begin{lemma}			\label{lem3.13}
Let $\eta$ be a Lipschitz function on $\bR^n$ such that $0 \le \eta \le 1$ and $\abs{D \eta} \le 4/\delta$ for some $\delta>0$.
Let
\[
\mathbf{g}=-\eta (\mathbf{A}-\mathbf{A}_0) G_0(\cdot, x_0)
\]
and take $y_0 \in \Omega \setminus \set{x_0}$ with $r:=\frac15 \abs{x_0-y_0} \le \delta$.
Then, we have for any $t \in (0, r]$ that
\[
\omega_{\mathbf{g}}(t, \Omega(y_0,2r)) \le C r^{2-n}\left( \omega_{\mathbf A}(t, \Omega(x_0, 7r))+ \frac{t}{r} \int_0^t \frac{\omega_{\mathbf A}(s, \Omega(x_0, 7r))}{s}\,ds\right),
\]
where $C=C(n, \lambda, \Lambda, \Omega)$.
\end{lemma}

The above lemma, the proof of which is given in Section~\ref{sec3.3}, yields that (take $\eta=\zeta$ with $\delta=r$)
\begin{align}				\nonumber
\int_0^{r} \frac{\omega_{\mathbf{g}_1}(t, \Omega(y_0,2r))}{t}\,dt &\le C r^{2-n} \left(\int_0^r \frac{\omega_{\mathbf A}(t)}{t}\,dt+\frac{1}{r} \int_0^r \int_0^t \frac{\omega_{\mathbf A}(s)}{s}\,ds\,dt \right)\\
						\label{eq0746f}
&\le C r^{2-n} \int_0^r \frac{\omega_{\mathbf A}(t)}{t}\,dt.
\end{align}
Putting \eqref{eq0746f} back to \eqref{eq2017th} together with \eqref{eq2018th}, we get
\begin{equation}			\label{eq0819f}
\abs{v_1(y_0)} \le  C r^{2-n}\left(1+ \int_0^{r} \frac{\omega_{\mathbf A}(t)}{t}\,dt \right). 
\end{equation}
Next, we shall estimate $v_2(y_0)$.
Again, by Lemma~\ref{prop02}, we have
\[
\abs{v_2(y_0)} \le  C \fint_{\Omega(y_0,2r)} \abs{v_2} + C \int_0^{r} \frac{\omega_{\mathbf{g}_2}(t, \Omega(y_0,2r))}{t}\,dt.
\]
Notice that $\mathbf{g}_2$ vanishes in $\Omega(y_0, 3r)$ since $B(x_0,2r) \cap B(y_0, 3r) =\emptyset$.
Therefore, we have $\omega_{\mathbf{g}_2}(t, \Omega(y_0,2r))=0$ and thus
\begin{align}				\nonumber
\abs{v_2(y_0)} & \le C \fint_{\Omega(y_0,2r)} \abs{v_2} \le C \left(\fint_{\Omega(y_0,2r)} \abs{v_2}^{p_2}\right)^{1/p_2}\\
						\label{eq1542th}
&\le C r^{-n/p_2} \norm{v_2}_{L^{p_2}(\Omega)} \le C\left(\int_0^{2r} \frac{\omega_{\mathbf A}(t)}{t}\,dt\right)^{1/p_2} r^{2-n},
\end{align}
where we used \eqref{eq1541th}.
Therefore, by using \eqref{eq0819f} and \eqref{eq1542th}, and recalling that $v=v_1+v_2$ and $r=\frac15 \abs{x_0-y_0}$, we have
\begin{equation}			\label{eq0824f}
\abs{v(y_0)} \le C \left(1+ \int_0^{\abs{x_0-y_0}} \frac{\omega_{\mathbf A}(t)}{t}\,dt + \left(\int_0^{2\abs{x_0-y_0}} \frac{\omega_{\mathbf A}(t)}{t}\,dt\right)^{\frac{1}{p_2}} \right)\, \abs{x_0-y_0}^{2-n},
\end{equation}
where $C=C(n,\lambda, \Lambda, \omega_{\mathbf A}, \Omega)$.
Since
\[
G^\ast(y_0,x_0)=G_0(y_0,x_0)+v(y_0)
\]
and $y_0 \in \Omega \setminus \set{x_0}$ is arbitrary, the desired estimate \eqref{eq0927f} follows from \eqref{eq1224th} and \eqref{eq0824f}.

\subsection{Proof of Lemma~\ref{lem3.13}}			\label{sec3.3}
For $\bar x \in \Omega(y_0, 2r)$ and $0<t \le r$, we have
\begin{align*}
\omega_{\mathbf{g}}(t,\bar x)&=\fint_{\Omega(\bar x,t)} \,\Abs{(\mathbf{A}-\mathbf{A}_0) G_0(\cdot, x_0) \eta- \overline{((\mathbf{A}-\mathbf{A}_0) G_0(\cdot, x_0) \eta)}_{\Omega(\bar x, t)}}\\
&\le \fint_{\Omega(\bar x,t)} \,\Abs{(\mathbf{A}-\mathbf{A}_0) G_0(\cdot, x_0) \eta- \overline{(\mathbf{A}-\mathbf{A}_0)}_{\Omega(\bar x, t)} G_0(\cdot, x_0) \eta}\\
&\qquad \quad+ \fint_{\Omega(\bar x,t)} \,\Abs{\overline{(\mathbf{A}-\mathbf{A}_0)}_{\Omega(\bar x, t)} G_0(\cdot, x_0) \eta- (\overline{(\mathbf{A}-\bar{\mathbf A}) G_0(\cdot, x_0) \eta})_{\Omega(\bar x, t)}}\\
&=:I+II.
\end{align*}

Observe that we have $\dist(x_0, \Omega(\bar x, t)) \ge 2r$ and thus for $x$, $y \in \Omega(\bar x,t)$, we have
\begin{equation}				\label{eq1225th}
\abs{G_0(x, x_0)- G_0(y, x_0)} \le C t r^{1-n},
\end{equation}
where $C=C(n,\lambda, \Lambda, \Omega)$.
Since $\dist(x_0, \Omega(\bar x, t)) \ge 2r$, by using \eqref{eq1224th} we obtain
\begin{align}				\nonumber
I &\le \fint_{\Omega(\bar x, t)} \,\Abs{(\mathbf{A}-\mathbf{A}_0)- \overline{(\mathbf{A}-\mathbf{A}_0)}_{\Omega(\bar x, t)}} \abs{G_0(\cdot, x_0)}\\
						\label{eq2222th}
&\le \fint_{\Omega(\bar x, t)} \,C r^{2-n} \abs{\mathbf{A}-\bar{\mathbf A}_{\Omega(\bar x,t)}} \le C r^{2-n} \omega_{\mathbf A}(t, \bar x).
\end{align}
Also, we have
\begin{align}				\nonumber
II &\le \fint_{\Omega(\bar x, t)} \,\Abs{\fint_{\Omega(\bar x, t)}(\mathbf{A}(y)-\mathbf{A}(x_0)) \left(G_0(x, x_0)\eta(x)- G_0(y, x_0)\eta(y)\right) dy} dx \\
						\label{eq2223th}
&\le \fint_{\Omega(\bar x, t)} \fint_{\Omega(\bar x, t)} \abs{\mathbf{A}(y)-\mathbf{A}(x_0)}\, \abs{G_0(x, x_0)\eta(x)- G_0(y, x_0)\eta(y)}\,dy\,dx.
\end{align}
By using \eqref{eq1224th}, \eqref{eq1225th}, $\abs{D \eta} \le 4/\delta$, and $r\le  \delta$, we have  for $x$, $y \in \Omega(\bar x, t)$ that
\begin{align}				\nonumber
\abs{G_0(x, x_0)\eta(x)- G_0(y, x_0)\eta(y)}& \le \abs{G_0(x, x_0)- G_0(y, x_0)}\,\abs{\eta(x)} + \abs{G_0(y, x_0)}\,\abs{\eta(x)- \eta(y)}\\
						\label{eq2224th}						
&\le C t r^{1-n}+C r^{2-n} t/\delta \le C t r^{1-n}.
\end{align}
Plugging \eqref{eq2224th} into \eqref{eq2223th}, we obtain
\[
II \le C t r^{1-n} \fint_{\Omega(\bar x, t)}  \abs{\mathbf{A}(y)-\mathbf{A}(x_0)}\,dy.
\]
We claim that
\begin{equation}			\label{eq1307th}
\fint_{\Omega(\bar x, t)}  \abs{\mathbf{A}(y)-\mathbf{A}(x_0)}\,dy \le C \left( \frac{r \omega_{\mathbf{A}}(t, \Omega(x_0, 7r))}{t}+\int^{t}_0 \frac{\omega_{\mathbf{A}}(s, \Omega(x_0, 7r))}{s}\, ds\right),
\end{equation}
where $C=C(n, \lambda, \Lambda, \Omega)$.
Let us take the claim granted for now.
Then, we have
\begin{equation}
						\label{eq2226th}						
II \le C t r^{1-n} \left( \frac{r \omega_{\mathbf{A}}(t, \Omega(x_0, 7r))}{t}+\int^{t}_0 \frac{\omega_{\mathbf{A}}(s, \Omega(x_0, 7r))}{s}\, ds\right).
\end{equation}
Combining \eqref{eq2222th} and \eqref{eq2226th}, we have (recall $t\le r$)
\begin{equation*}			
\omega_{\mathbf{g}}(t, \bar x) \le I+II \le C r^{2-n}\left( \omega_{\mathbf A}(t, \Omega(x_0, 7r))+ \frac{t}{r} \int_0^t \frac{\omega_{\mathbf A}(s, \Omega(x_0, 7r))}{s}\,ds\right).
\end{equation*}
The lemma is proved by taking supremum over $\bar x \in \Omega(y_0, 2r)$.

It remains to prove the claim \eqref{eq1307th}.
Notice that we can choose a sequence of points $x_1$, $x_2$, $\ldots$, $x_N$ in $\Omega(x_0, 7r)$ with $x_N=\bar x$ in such a way that each line segment $[x_{i-1}, x_i]$ lies in $\Omega$ and $\abs{x_{i-1}-x_i} \le t$ for $i=1,\ldots, N$.
Moreover, there exists a constant $C=C(\Omega)$ independent of $t$ and $r$ such that
\begin{equation}				\label{eq2202th}
N \le Cr/t.
\end{equation}
Then by using triangle inequalities, we have
\begin{equation}			\label{eq1712th}
\abs{\mathbf{A}(y)-\mathbf{A}(x_0)} \le \abs{\mathbf{A}(y)-\bar{\mathbf A}_{\Omega(\bar x,t)}} + \sum_{i=1}^N \,\abs{\bar{\mathbf A}_{\Omega(x_i,t)}-\bar{\mathbf A}_{\Omega(x_{i-1},t)}} + \abs{\bar{\mathbf A}_{\Omega(x_0,t)}-\mathbf{A}(x_0)}.
\end{equation}
Note that by \cite[Appendix]{HK20}, we have
\begin{equation}			\label{eq1713th}
\begin{aligned}
\abs{\bar{\mathbf A}_{\Omega(x_0,t)}-\mathbf{A}(x_0)} &\le C\int_0^t \frac{\omega_{\mathbf A}(s,x_0)}{s}\,ds,\\
\abs{\bar{\mathbf A}_{\Omega(x_i,t)}-\bar{\mathbf A}_{\Omega(x_{i-1},t)}} &\le C \omega_{\mathbf A}(t, \Omega(x_0, 7r)),  \quad i=1,\ldots, N.
\end{aligned}
\end{equation}
Using \eqref{eq1713th} and averaging the inequality \eqref{eq1712th} over $y \in \Omega(\bar x, t)$, we obtain
\[
\fint_{\Omega(\bar x, t)}  \abs{\mathbf{A}(y)-\mathbf{A}(x_0)}\,dy \le \omega_{\mathbf A}(t, \bar x)+ CN \omega_{\mathbf A}(t, \Omega(x_0, 7r)) + C\int_0^t \frac{\omega_{\mathbf A}(s,x_0)}{s}\,ds.
\]
Then \eqref{eq1307th} follows from the above inequality and \eqref{eq2202th}. \qed

\subsection{Construction and symmetry relation for Green's function}			\label{sec3.4}
In this section, we shall prove that 
The fucntion $G(x,y)$ given by
\begin{equation}		\label{symmetry}
G(x,y)=G^\ast(y,x),\quad \forall\, x, y \in \Omega, \quad x \neq y,
\end{equation}
is the Green's function for the operator $L$ in $\Omega$.
Then in light of \eqref{eq0927f}, we see that the Green's function $G(x,y)$ has the pointwise bound \eqref{bound1}.

To establish \eqref{symmetry}, first observe that $G^\ast(\cdot, x_0)$ satisfies
\[
L^\ast G^\ast(\cdot, x_0)=0\;\text{ in }\;\Omega\setminus B(x_0, r)\;\text{ for any }\;r>0.
\] 
By \cite{DK17, DEK18}, we see that $G^\ast(\cdot, x_0)$ is continuous in $\Omega\setminus B(x_0, r)$ for any $r>0$.
Next, for $y_0 \in \Omega$ and $\epsilon>0$, let $u=G_\epsilon(\cdot, y_0) \in W^{2,p}(\Omega) \cap W^{1,p}_0(\Omega)$ be a unique strong solution of the problem \eqref{eq2059m} with $f=\frac{1}{\abs{\Omega(y_0,\epsilon)}} \, \chi_{\Omega(y_0,\epsilon)}$.
Then by \eqref{eq2037m}, we have
\begin{equation}				\label{eq1709th}
G_\epsilon(x_0, y_0)= \fint_{\Omega(y_0, \epsilon)} G^\ast(y, x_0)\,dy.
\end{equation}
We conclude from \eqref{eq1709th} and \eqref{eq0927f} that for any $x$, $y \in \Omega$ with $x\neq y$, we have
\[
\abs{G_\epsilon(x,y)} \le C \abs{x-y}^{2-n}, \quad \forall  \epsilon \in(0, \tfrac13 \abs{x-y}),
\]
which coincides with \cite[Lemma~2.11]{HK20}.
With the above key estimate at hand, we can replicate the same argument as in \cite{HK20} and construct the Green's function $G(x,y)$ for the operator $L$ out of the family $\set{G_\epsilon(x,y)}$.
In particular, there is a sequence $\set{\epsilon_j} \to 0$ such that (see \cite[(2.24)]{HK20})
 \[
G_{\epsilon_j}(\cdot, y_0) \to G(\cdot, y_0)\;\text{ uniformly on }\; \Omega\setminus  B(y_0,r),\quad  \forall\, r>0 .
\]
Then, by using the continuity of $G^\ast(\cdot, x_0)$ away from $x_0$, we derive from \eqref{eq1709th} the desired identity \eqref{symmetry}.
\qed

\subsection{Asymptotic behavior near a pole}			\label{sec3.5}
In this section, we assume that condition \eqref{eq1243f} holds for some $r_0>0$.
The estimate \eqref{eq0824f} leaves a room that we might be able to get an asymptotic behavior
\[
G^\ast(x,x_0)-G_0(x, x_0)= o(\abs{x-x_0}^{2-n})\;\text{ as }\; x\to x_0.
\]
To see this, we closely follow the argument in Section~\ref{sec3.2}.
Let $v$ be as before in \eqref{eq1042th}. 
For any $\epsilon>0$, we can choose $\delta \in (0, \frac18 r_0]$ such that 
\begin{equation}			\label{eq1349f}
\left(\int^{\delta}_0 \frac{1}{s}\int_0^s \frac{\omega_{\mathbf A}(t, \Omega(x_0, r_0))}{t}\,dt\,ds\right)^{\frac{n-2}{n}} < \epsilon.  
\end{equation}  
Let $\zeta$ be a smooth function on $\bR^n$ such that
\[
0\le \zeta \leq 1, \quad \zeta=0 \;\text{ in }\;B(x_0,\delta/2),\quad \zeta=1\;\text{ in }\bR^n\setminus B(x_0,\delta),\quad \abs{D \zeta} \le 4/\delta,
\]
We then define $\mathbf{g}_1$ and $\mathbf{g}_2$ by
\begin{equation*}
\mathbf{g}_1= -\zeta(\mathbf{A}- \mathbf{A}_0) G_0(\cdot, x_0)\quad\text{and}\quad \mathbf{g}_2= -(1-\zeta)(\mathbf{A}-\mathbf{A}_0) G_0(\cdot, x_0).
\end{equation*}
Then we have  $\mathbf{g}_1 \in L^{p}(\Omega)$ for $p > \frac{n}{n-2}$ and $\mathbf{g}_2 \in L^{\frac{n}{n-2}}(\Omega)$.
Indeed, we have
\begin{equation}				\label{eq1730f}
\int_\Omega \abs{\mathbf{g}_1}^{p} \le C \int_{\Omega\setminus \Omega(x_0,\delta/2)} \abs{x-x_0}^{(2-n)p}\,dx \le C\delta^{(2-n)p+n}
\end{equation}
and using \eqref{eq1224th}, \eqref{eq1142tu}, and \eqref{eq1349f}, we have 
\begin{align}					
						\nonumber
\int_\Omega\, \abs{\mathbf{g}_2}^{\frac{n}{n-2}}& = \sum_{k=0}^\infty \int_{\Omega(x_0, 2^{-k}\delta)\setminus \Omega(x_0, 2^{-k-1}\delta)} \abs{\mathbf{g}_2}^{\frac{n}{n-2}} \le C \sum_{k=0}^\infty  \fint_{\Omega(x_0, 2^{-k}\delta)} \abs{\mathbf{A}-\mathbf{A}_0} \\
						\label{eq1432f}
&\le C \sum_{k=0}^\infty \int_0^{2^{-k}\delta} \frac{\omega_{\mathbf A}(t)}{t}\,dt \le  C\left(\int_0^{\delta} \frac{1}{s}\int_0^s \frac{\omega_{\mathbf A}(t,x_0)}{t}\,dt\,ds \right) \le C\epsilon^{\frac{n}{n-2}}.
\end{align}

Let $v_1$ and $v_2$ be the solutions of the problems \eqref{eq1226th} and \eqref{eq1244th}, respectively.
Then, similar to \eqref{eq1618th} and \eqref{eq1541th}, using \eqref{eq1730f} and \eqref{eq1432f}, we obtain
\begin{equation}				\label{eq1618f}
\norm{v_1}_{L^{p}(\Omega)} \le C\delta^{(2-n)+\frac{n}{p}}\quad (\,p>\tfrac{n}{n-2}\,)\qquad\text{and}\qquad
\norm{v_2}_{L^{\frac{n}{n-2}}(\Omega)} \le C\epsilon.
\end{equation}
For $y_0 \in \Omega \setminus \set{x_0}$ with $\abs{y_0-x_0} \le 5\delta$, we estimate 
$v_1(y_0)$ and $v_2(y_0)$ as follows.
Set
\[
r=\tfrac15 \abs{y_0-x_0}
\]
and using Lemma~\ref{prop02}, we have
\begin{equation}				\label{eq1335sat}
\abs{v_i(y_0)} \le  C \fint_{\Omega(y_0,2r)} \abs{v_i} + C \int_0^{r} \frac{\omega_{\mathbf{g}_i}(t, \Omega(y_0,2r))}{t}\,dt,\quad i=1,2.
\end{equation}
Using \eqref{eq1618f} together with H\"older's inequalities, we have
\begin{equation}				\label{eq1340sat}
\begin{aligned}
\fint_{\Omega(y_0,2r)} \abs{v_1} &\le C r^{-\frac{n}{p}}  \norm{v_1}_{L^{p}(\Omega(y_0, 2r))} \le C  r^{-\frac{n}{p}}\delta^{(2-n)+\frac{n}{p}}\quad (\,p>\tfrac{n}{n-2}\,),\\
\fint_{\Omega(y_0,2r)} \abs{v_2} &\le Cr^{2-n}  \norm{v_2}_{L^{\frac{n}{n-2}}(\Omega(y_0, 2r))} \le C \epsilon r^{2-n}.
\end{aligned}
\end{equation}
On the other hand, applying Lemma~\ref{lem3.13} with $\eta=\zeta$ and $\eta=1-\zeta$, respectively, and using the fact that $r\le \delta \le \frac18 r_0$, for all $t \in (0,r]$ we have
\begin{multline*}
\omega_{\mathbf{g}_i}(t, \Omega(y_0,2r))
 \le C r^{2-n}\left( \omega_{\mathbf A}(t, \Omega(x_0, r_0))+ \frac{t}{r} \int_0^t \frac{\omega_{\mathbf A}(s, \Omega(x_0, r_0))}{s}\,ds\right),\quad i=1,2.
\end{multline*}
Then, similar to \eqref{eq0746f}, we have
\begin{equation}				\label{eq1331sat}
\int_0^{r} \frac{\omega_{\mathbf{g}_i}(t, \Omega(y_0,2r))}{t}\,dt \le C r^{2-n} \int_0^r \frac{\omega_{\mathbf A}(t, \Omega(x_0, r_0))}{t}\,dt,\quad i=1,2.
\end{equation}
Now we substitute \eqref{eq1340sat} and \eqref{eq1331sat} back to \eqref{eq1335sat} to obtain
\begin{align}					\nonumber
\abs{v(y_0)}&\le \abs{v_1(y_0)}+\abs{v_2(y_0)} \\
							\label{eq1406sat}
&\le C r^{2-n}\left( r^{n-2-\frac{n}{p}}\delta^{(2-n)+\frac{n}{p}}+ \epsilon +\int_0^r \frac{\omega_{\mathbf A}(t, \Omega(x_0, r_0))}{t}\,dt\right)\quad (\,p>\tfrac{n}{n-2}\,).
\end{align}
Note that $p>\frac{n}{n-2}$ implies $n-2-\frac{n}{p}>0$.
From \eqref{eq1406sat} and the fact that
\[
v=G^*(\cdot,x_0)-G_0(\cdot,x_0)\quad\text{and}\quad r=\tfrac15 \abs{y_0-x_0},
\]
we conclude that
\[
\lim_{x\to x_0}\, \abs{x-x_0}^{n-2} \,\abs{G^*(x,x_0)-G_0(x,x_0)} =0
\]
since $y_0 \in \Omega \setminus \set{x_0}$ and $\epsilon$ are arbitrary.
Since $G_0$ is symmetric, we obtain \eqref{eq1725sun} from the above and \eqref{symmetry}.

\section{Appendix: Proof of Lemma~\ref{prop01}}
Let us consider the quantity
\[
\phi(x,r):=\inf_{q\in \bR}\left( \fint_{B(x,r)}
\abs{u - q}^{\frac12} \right)^{2}
\]
for $x \in B(x_0, \frac32 R)$ and $0<r\le \frac14 R$.
We decompose $u=v+w$, where $w \in L^{2}(B(x, r))$ is the solution of the problem
\[
\left\{
\begin{aligned}
L_0^*w &= -\dv^2 \left((\mathbf{A}- \bar{\mathbf A}) u\right) +\dv^2 \left(\mathbf{g}-\bar{\mathbf g}\right)\;\mbox{ in }\;B(x, r),\\
w&=\frac{\left(\mathbf{g}-\bar{\mathbf g}-(\mathbf{A}- \bar{\mathbf A}) u\right)\nu\cdot \nu}{\bar{\mathbf A}\nu\cdot \nu} \;\mbox{ on }\;\partial B(x, r),
\end{aligned}
\right.
\]
where we use the notation
\[
\bar{\mathbf A}=\bar{\mathbf A}_{B(x, r)},\quad \bar{\mathbf g} = \bar{\mathbf g}_{B(x, r)},\quad  L_0^\ast w=  \dv^2(\bar{\mathbf A} w).
\]
By \cite[Lemma~2.23]{DK17}, we have
\[
\abs{\set{y\in B(x, r): \abs{w(y)} > t}}
\le \frac{C}{t}\left(\norm{u}_{L^\infty(B(x, r))} \int_{B(x, r)} \abs{\mathbf{A}-\bar{\mathbf A}} +  \int_{B(x, r)}  \abs{\mathbf g -\bar{\mathbf g}} \,\right),
\]
where $C=C(n,\lambda, \Lambda)$.
This yields (see \cite[pp. 422-423]{DK17} and recall the notation \eqref{def_mo})
\begin{equation}				\label{eq15.50a}
\left(\fint_{B(x, r)} \abs{w}^{\frac12} \right)^{2} \le C\omega_{\mathbf A}(r) \,\norm{u}_{L^\infty(B(x, r))} + C \omega_{\mathbf g}(r, x),
\end{equation}
where $C=C(n,\lambda, \Lambda)$.
On the other hand, note that $v=u-w$ satisfies
\[
L_0^\ast  v = \dv^2(\bar{\mathbf A} v)= \dv(\bar{\mathbf A} \nabla v)= 0 \;\mbox{ in }\;B(x, r),
\]
and so does $v-q$ for any constant $q \in \bR$.
By the interior estimates for elliptic equations with constant coefficients, we have
\[
\norm{D v}_{L^\infty(B(x, \frac12 r))} \le C_0 r^{-1} \left(\fint_{B(x, r)} \abs{v - q}^{\frac12}\,\right)^{2},
\]
where $C_0=C_0(n,\lambda, \Lambda)>0$ is a constant.
Let $0< \kappa \le \frac12$ to be a number to be fixed later.
Then, we have
\begin{equation}				\label{eq15.50b}
\left(\fint_{B(x, \kappa r)} \abs{v - \overline{v}_{B(x, \kappa r)}}^{\frac12} \right)^{2} \le 2\kappa r \norm{D v}_{L^\infty(B(x, \frac12 r))} \le 2C_0 \kappa \left(\fint_{B(x, r)} \abs{v - q}^{\frac12} \, \right)^{2}.
\end{equation}
By using the decomposition $u=v+w$, we obtain from \eqref{eq15.50b} that
\begin{align*}
\left(\fint_{B(x, \kappa r)} \abs{u - \overline{v}_{B(x, \kappa r)}}^{\frac12}\right)^{2} & \le 2 \left(\fint_{B(x, \kappa r)} \abs{v - \overline{v}_{B(x, \kappa r)}}^{\frac12} \right)^{2}+ 2\left(\fint_{B(x, \kappa r)} \abs{w}^{\frac12} \right)^{2}\\
& \le 4C_0\kappa \left(\fint_{B(x, r)} \abs{v - q}^{\frac12} \right)^{2}+ 2\left(\fint_{B(x, \kappa r)} \abs{w}^{\frac12} \right)^{2}\\
& \le  8C_0\kappa \left(\fint_{B(x, r)} \abs{u - q}^{\frac12} \right)^{2} + (2\kappa^{-2n}+8C_0 \kappa) \left(\fint_{B(x, r)} \abs{w}^{\frac12} \right)^{2}.
\end{align*}
Since $q\in \bR$ is arbitrary, by using \eqref{eq15.50a}, we thus obtain
\[
\phi(x,\kappa r) \le 8C_0 \kappa \, \phi(x, r)+ C \left( \omega_{\mathbf A}(r) \,\norm{u}_{L^\infty(B(x, r))} +  \omega_{\mathbf g}(r, x) \right),
\]
where $C=C(n,\lambda, \Lambda, \kappa)$.
Now we choose $\kappa$ such that $8C_0\kappa=\frac12$. 
Then we have
\[
\phi(x, \kappa r) \le \frac12 \phi(x,r)+ C \left( \omega_{\mathbf A}(r) \,\norm{u}_{L^\infty(B(x,r))} +  \omega_{\mathbf g}(r, x) \right),
\]
where $C=C(n,\lambda, \Lambda)$.
By iterating, for $j=1,2,\ldots$, we get
\[
\phi(x, \kappa^j r) \le 2^{-j} \phi(x,r) +C \norm{u}_{L^\infty(B_r(x))} \sum_{i=1}^{j} 2^{1-i} \omega_{\mathbf A}(\kappa^{j-i} r) + C \sum_{i=1}^{j} 2^{1-i} \omega_{\mathbf g}(\kappa^{j-i} r, x).
\]
We note that
\begin{align*}			
\sum_{j=0}^\infty \sum_{i=1}^{j} 2^{1-i} \omega_{\mathbf A}(\kappa^{j-i} r)&=\sum_{i=1}^\infty \sum_{j=i}^\infty 2^{1-i}\omega_{\mathbf A}(\kappa^{j-i} r)=\sum_{i=1}^\infty  2^{1-i} \sum_{j=0}^\infty \omega_{\mathbf A}(\kappa^{j} r)\\
&= 2 \sum_{j=0}^\infty \omega_{\mathbf A}(\kappa^{j} r)\le C \int_0^{r} \frac{\omega_{\mathbf A}(t)}{t}\,dt<+\infty,
\end{align*}
where $C=C(\kappa)=C(n,\lambda, \Lambda)$ and we used \cite[Lemma~2.7]{DK17}.
A similar computation holds for $\omega_{\mathbf g}(r, x)$, and thus  we obtain
\begin{equation}		\label{eq2013th}
\sum_{j=0}^\infty \phi(x, \kappa^j r) \le 2\phi(x, r) + C \norm{u}_{L^\infty(B(x,r))} \int_0^{r} \frac{\omega_{\mathbf A}(t)}{t}\,dt + C \int_0^{r} \frac{\omega_{\mathbf g}(t, x)}{t}\,dt.
\end{equation}


Now, let $q_{x,r}$ be chosen so that
\begin{equation}		\label{eq12.14fa}
\left(\fint_{B_r(x)} \abs{u - q_{x,r}}^{\frac12} \right)^2 =\inf_{q \in \bR} \left( \fint_{B_r(x)} \abs{u - q}^{\frac12} \right)^2=\phi(x, r).
\end{equation}
Since we have
\[
\abs{q_{x,r} - q_{x, \kappa r}}^{\frac12} \le
\abs{u(y)- q_{x,r}}^{\frac12} + \abs{u(y) - q_{x, \kappa r}}^{\frac12},
\]
taking the average over $y \in B_{\kappa r}(x)$ and then taking the square, we obtain
\[
\abs{q_{x, r} - q_{x,\kappa r}}\le 2 \kappa^{-2n} \phi(x, r) +2 \phi(x,\kappa r) \le 2 \kappa^{-2n} \left( \phi(x, r)+ \phi(x, \kappa r)\right).
\]
Then, by iterating and using the triangle inequality
\[
\abs{q_{x,\kappa^{N} r} - q_{x, r}}  \le 4 \kappa^{-2n} \sum_{j=0}^N \phi(x, \kappa^{j} r).
\]
Therefore, by using the fact that $q_{x,\kappa^{N} r} \to u(x)$ as $N\to \infty$ and \eqref{eq2013th}, we have 
\begin{align}
					\nonumber
\abs{u(x)-q_{x,r}}  & \le 4 \kappa^{-2n} \sum_{j=0}^\infty \phi(x, \kappa^{j} r) \\							\label{eq13.13}
&\le C \phi(x, r)+C\norm{u}_{L^\infty(B(x,r))} \int_0^{r} \frac{\omega_{\mathbf A}(t)}{t}\,dt + C \int_0^{r} \frac{\omega_{\mathbf g}(t, x)}{t}\,dt.
\end{align}
By averaging the inequality
\[
\abs{q_{x, r}}^{\frac12} \le \abs{u(y) -q_{x,r}}^{\frac12} + \abs{u(y)}^{\frac12}
\]
over $y \in B(x,r)$, taking the square, and using \eqref{eq12.14fa} we get
\[
\abs{q_{x, r}} \le 2 \phi(x,r) + 2 \left(\fint_{B(x, r)} \abs{u}^{\frac12}\right)^{2} \le 4 r^{-n} \norm{u}_{L^1(B(x, r))}.
\]
Therefore, by combining the above with \eqref{eq13.13}, we get
\[
\abs{u(x)}  \le C r^{-n} \norm{u}_{L^1(B(x,r))} + C\norm{u}_{L^\infty(B(x,r))} \int_0^{r} \frac{\omega_{\mathbf A}(t)}{t}\,dt + C \int_0^{r} \frac{\omega_{\mathbf g}(t, x)}{t}\,dt.
\]
Now, taking supremum for $x\in B(\bar x,r)$, where $\bar x \in B(x_0, \frac32 R)$ and $0<r \le \frac14 R$, we have
\[
\norm{u}_{L^\infty(B(\bar x,r))}
\le C r^{-n} \norm{u}_{L^1(B(\bar x, 2r))} + C \norm{u}_{L^\infty(B(\bar x, 2r))} \int_0^{r} \frac{\omega_{\mathbf A}(t)}{t}\,dt +C \int_0^{r} \frac{\omega_{\mathbf g}(t, B(\bar x,r))}{t}\,dt.
\]
We fix $r_0=r_0(n,\lambda, \Lambda, \omega_{\mathbf A})$ such that
\[
C \int_0^{r_0} \frac{\omega_{\mathbf A}(t)}{t}\,dt \le \frac1{3^{n}}.
\]
Then, we have for any $\bar x \in B(x_0, \frac32 R)$ and $0<r\le \min(r_0, \frac14 R)$ that
\begin{equation}				\label{eq1716tu}
\norm{u}_{L^\infty(B(\bar x,r))} \le
3^{-n}\norm{u}_{L^\infty(B(\bar x,2r))} + C r^{-n} \norm{u}_{L^1(B(x_0,2R))} + C \int_0^{r} \frac{\omega_{\mathbf g}(t, B(x_0,2R))}{t}\,dt.
\end{equation}

For $k=1,2,\ldots$, denote 
\[r_k=\left(\frac32-\frac{1}{2^k}\right)R.\]
Note that $r_{k+1}-r_k=2^{-k-1}R$ and $r_1=R$.
For $\bar x\in B(x_0, r_k)$ and $r\le 2^{-k-2}R$, we have $B(\bar x, 2r) \subset B(x_0,r_{k+1})$.
We take $k_0\ge 1$ sufficiently large such that $2^{-k_0-3} R_0\le r_0$.
Note that $k_0=k_0(r_0, R_0)=k_0(n, \lambda, \Lambda, \omega_{\mathbf A}, R_0)$.
Then for any $k \ge k_0$, we have $2^{-k-2}R \le r_0$ and thus by taking $r=2^{-k-2}R$ in \eqref{eq1716tu}, we obtain
\[
\norm{u}_{L^\infty(B(x_0, r_k))} \le 3^{-n} \norm{u}_{L^\infty(B(x_0, r_{k+1}))}+C 2^{kn} R^{-n} \norm{u}_{L^1(B(x_0,2R))}+C \int_0^{R} \frac{\omega_{\mathbf g}(t, B(x_0,2R))}{t}\,dt.
\]
Multiplying the above by $3^{-kn}$ and summing over $k=k_0, k_0+1,\ldots$, we get
\begin{multline*}			
\sum_{k=k_0}^\infty 3^{-kn}\norm{u}_{L^\infty(B(x_0,r_k))} \le \sum_{k=k_0}^\infty 3^{-(k+1)n} \norm{u}_{L^\infty(B(x_0,r_{k+1}))}\\
+C R^{-n}\norm{u}_{L^1(B(x_0, 2R))} + \int_0^{R} \frac{\omega_{\mathbf g}(t, B(x_0,2R))}{t}\,dt.
\end{multline*}
Noting that $R=r_1 \le r_{k_0}$, and shifting the index in the second sum, we get \eqref{eq1023m}.\qed


\end{document}